\documentclass{amsart}
\usepackage{amssymb}
\usepackage{xypic}
\xyoption{all}
\usepackage{url}

\title[On cohomological triviality of the center of the Frattini subgroup]{On the cohomological triviality of the center of the Frattini subgroup}

\author[J. Calles]{Jaime Calles}
\thanks{}
\address{
\hfill\break Universidad Nacional Aut\'onoma de M\'exico \\
\hfill\break Instituto de Matem\'aticas \\
\hfill\break Investigaci\'on Cient\'ifica, C.U. \\
\hfill\break Coyoac\'an, Ciudad de M\'exico, CDMX 04510 \\
\hfill\break Mexico.}
\email{calles@im.unam.mx}

\author[J. Cantarero]{Jos\'e Cantarero}
\thanks{}
\address{
\hfill\break Centro de Investigaci\'on en Matem\'aticas, A.C., Unidad M\'erida \\
\hfill\break Parque Cient\'ifico y Tecnol\'ogico de Yucat\'an  \\ 
\hfill\break Carretera Sierra Papacal--Chuburn\'a Puerto Km 5.5 \\
\hfill\break Sierra Papacal, M\'erida, YUC 97302 \\
\hfill\break Mexico.}
\email{cantarero@cimat.mx}

\author[J. O. G\'omez]{Juan Omar G\'omez}
\thanks{}
\address{
\hfill\break Centro de Investigaci\'on en Matem\'aticas, A.C., Unidad M\'erida \\
\hfill\break Parque Cient\'ifico y Tecnol\'ogico de Yucat\'an  \\ 
\hfill\break Carretera Sierra Papacal--Chuburn\'a Puerto Km 5.5 \\
\hfill\break Sierra Papacal, M\'erida, YUC 97302 \\
\hfill\break Mexico.}
\email{juan.gomez@cimat.mx}

\author[G. Ortega]{Gustavo Ortega}
\thanks{}
\address{
\hfill\break Universidad Nacional Aut\'onoma de M\'exico \\
\hfill\break Facultad de Ciencias \\
\hfill\break Investigaci\'on Cient\'ifica, C.U. \\
\hfill\break Coyoac\'an, Ciudad de M\'exico, CDMX 04510 \\
\hfill\break Mexico.}
\email{gustavo.ortega@ciencias.unam.mx}

\newcommand{\comments}[1]{}

\newcommand{\Aut}{\operatorname{Aut}\nolimits}

\newcommand{\Mod}[1]{\ \mathrm{mod}\ #1}

\newcommand{\res}{\operatorname{res}\nolimits}

\def \F{{\mathbb F}}

\def \Z{{\mathbb Z}}

\theoremstyle{plain}
\newtheorem*{introtheorem}{Theorem}
\newtheorem{theorem}{Theorem}[section]
\newtheorem{proposition}[theorem]{Proposition}
\newtheorem{corollary}[theorem]{Corollary}
\newtheorem{lemma}[theorem]{Lemma}

\theoremstyle{definition}

\newtheorem{remark}[theorem]{Remark}



\begin{document}


\begin{abstract}
We improve the existing lower bounds on the order of counterexamples to a conjecture
by P. Schmid, determine some properties of the possible counterexamples of minimum
order for each prime, and the isomorphism type of the center of the Frattini subgroup for the counterexamples
of order $256$. We also show that nonabelian metacyclic $p$-groups, nonabelian groups of
maximal nilpotency class and $2$-groups of coclass two satisfy the conjecture.
\end{abstract}

\maketitle

\section*{Introduction}
\label{sec : Introduction}

In 1973, Ya. G. Berkovich conjectured that every finite nonabelian $p$-group admits a
noninner automorphism of order $p$. This is Problem 4.13 from the 4th issue of The Kouravka
Notebook \cite{KM}. The existence of noninner automorphisms of $p$-power order had been
shown previously in \cite{Ga}.

This problem can be attacked by cohomological methods. In \cite{Sch}, P. Schmid showed
that if $G$ is a regular nonabelian finite $p$-group and $\Phi(G)$ is its Frattini subgroup,
then $Z(\Phi(G))$ is not cohomologically trivial over the Frattini quotient of $G$, and
in this case this implies the existence of a noninner automorphism of order $p$ that 
fixes the Frattini subgroup elementwise. The article conjectured that $Z(\Phi(G))$ is not a 
cohomologically trivial module over the Frattini quotient of $G$ for any finite nonabelian 
$p$-group $G$.

This conjecture turned out to be false. In \cite{Ab2}, A. Abdollahi found ten counterexamples of order $256$
using GAP. Hence the direction of this problem shifted towards proving the conjecture
for several families of groups, such as semi-abelian groups \cite{BG} and groups with small nilpotency
class (\cite{Ab07}, \cite{Ab2}, \cite{AGG}). Nonabelian $p$-groups are called S--groups 
if they satisfy the conjecture and NS--groups otherwise.

In this article we contribute to the understanding of this problem in two directions. 
First, we improve the existing lower bound on the order of possible NS--groups,
giving now new bounds that depend on $p$. We show that the minimum order of an NS--group 
is at least $p^{p+6}$ and find additional properties that an NS--group of order $p^{p+6}$ 
should satisfy. In the case of order $256$, we are able to identify the center of the Frattini
subgroup up to isomorphism, by a deeper analysis of cohomologically trivial modules
over $p$-groups which are finite $p$-groups. The following theorem combines the most
important results of Sections \ref{sec : SmallOrder} and \ref{sec : Faithful}.

\begin{introtheorem}
If $G$ is an NS--group, then $|G| \geq p^{p+6}$. Moreover, if $|G|=p^{p+6}$, then
\begin{itemize}
\item[(a)] $d(G)=2$
\item[(b)] $|Z(\Phi(G))|=p^{p+2}$
\item[(c)] $Z(G)$ is cyclic of order $p$.
\item[(d)] If $p=2$, then $Z(\Phi(G))$ is isomorphic to $(\Z/2)^4$ or $(\Z/4) \times (\Z/2)^2$.
\end{itemize}
\end{introtheorem}

In fact, we are able to improve these lower bounds for NS--groups that are $m$-generated
with $m>2$, but no such groups are known so far. The ten NS--groups known at the moment
are all $2$-generated. In the computations carried out in Section \ref{sec : Faithful}, we found the following
result which may be of independent interest.

\begin{introtheorem}
Let $M$ be a finite abelian $p$-group with an effective action of an elementary abelian $p$-group $E$
of rank at least two. If $M$ is cohomologically trivial, then
\begin{itemize}
\item $M$ is not cyclic.
\item $M$ is not isomorphic to $\Z/p^k \times \Z/p$ for any $k \geq 2$. 
\item $M$ is not isomorphic to $(\Z/p)^k$ unless $k$ is a multiple of $p^r$, where $r$ is the rank of $E$.
\item If $p=2$, then $M$ is not isomorphic to $\Z/4 \times \Z/4$. 
\end{itemize}
\end{introtheorem}

Finally, in Section \ref{sec : SmallCoclass} we determine some new families of S--groups, inspired
by results on Berkovich's conjecture for groups of small coclass and the method of determining the
trace map for metacyclic $2$-groups. The following theorem outlines the main results of that section.

\begin{introtheorem}
Any group in one of the following families is an $S$--group.
\begin{itemize}
\item Nonabelian maximal nilpotency class $p$-groups.
\item $2$-groups of coclass two.
\item Nonabelian metacyclic $p$-groups.
\end{itemize}
\end{introtheorem}

Given the close connection between Berkovich's and Schmid's conjectures and the recent results
about Berkovich's conjecture on groups of coclass at most three \cite{RLY}, it seems possible that all groups
in this family are S--groups. However, our methods do not apply to $p$-groups of coclass two for
$p \neq 2$ nor to $p$-groups of coclass three. And it is worth noting that Berkovich's conjecture is 
still unsettled for $3$-groups of coclass three.
\newline

\noindent \begin{bf}Notation and terminology.\end{bf} Given a finite group $G$ and a $\Z G$--module $M$, 
we denote by $\hat{H}^n(G;M)$ the $n$th Tate cohomology group of $G$ with coefficients in $M$. It is convenient 
to introduce some terminology to describe the behaviour of nonabelian $p$-groups with respect to the question 
raised by P. Schmid from the remark in page 3 of \cite{Sch}. Namely, we say that a nonabelian $p$-group $G$ is an 
S--group if
\[ \hat{H}^0(G/\Phi(G);Z(\Phi(G))) \neq 0 \]
for the conjugation action of the Frattini quotient $G/\Phi(G)$ on $Z(\Phi(G))$. Otherwise,
we say that $G$ is an NS--group. Note that by a theorem of Gasch\"utz and Uchida (see for
instance Theorem 4 in \cite{Uch}), a nonabelian $p$-group $G$ is an NS--group if and only 
if $Z(\Phi(G))$ is a cohomologically trivial $G/\Phi(G)$--module.

Given a $\Z Q$--module $A$, we denote by $[A,Q]$ the subgroup of $A$ generated by elements of the form
$ q\cdot a - a$. Since $[A,Q]$ is a $Q$--submodule of $A$, we can iterate this construction and 
define inductively  
\[ [A,\underbrace{Q,\ldots,Q}_{n \text{ times}}] = [[A,\underbrace{Q,\ldots,Q}_{n-1 \text{ times}}],Q] \]
We denote by $A^Q$ the submodule of fixed points and use $A_Q = A/[A,Q]$ for the coinvariants. Note
that our notation for invariants and coinvariants differs from the notation in
\cite{Ab2}, \cite{Sch} and related articles. We refer to the map $A_Q \to A^Q$
induced by $ a \mapsto \sum_{q \in Q} qa$ as the trace. The rest of notation that we use
is standard in group theory and group cohomology. 
\newline

\noindent \begin{bf}Acknowledgments.\end{bf} This project was seeded by the program ``Research groups''
that took place in the period January-June 2022 at CIMAT M\'erida's Algebraic Topology Seminar.

\section{The conjecture for groups of small order}
\label{sec : SmallOrder}

In this section we find a lower bound for NS--groups that depends on the minimum number
of generators of the group. In particular, we show that Schmid's conjecture holds for 
nonabelian $p$-groups of order at most $p^{p+5}$. 

We begin with the following proposition, which is implicit in the proof of Theorem 1.3 in
\cite{Gh2}.

\begin{proposition}
\label{Frattini nonabelian}
Let $G$ be a finite nonabelian $p$-group. If $\Phi(G)$ is abelian, then $G$ is an $S$--group. 
\end{proposition}

\begin{proof}
Assume that $G$ is an NS--group. Since $H^2(G/\Phi(G);\Phi(G))=0$, the group $G$ is the semidirect 
product of $G/\Phi(G)$ by $\Phi(G)$. But the elements of $\Phi(G)$ are non-generators,
hence the quotient $G \to G/\Phi(G)$ is an isomorphism. Then $\Phi(G)$ is trivial, which contradicts
the fact that $G$ is nonabelian.
\end{proof}

\begin{theorem}
\label{SmallFrattini}
If $G$ is an NS--group with $d(G)=m$, then $|\Phi(G)| \geq p^{m+p+2}$.
\end{theorem}

\begin{proof}
Let $A=Z(\Phi(G))$ and $Q=G/\Phi(G)$. Since $A$ is a cohomologically trivial $Q$--module,
\[ |A| = |A^Q| \cdot |A^Q \otimes Q| \cdot | [A,Q,Q] | \]
by Corollary 2.2 in \cite{AGG}. Let $A(n) = [A,Q,\ldots,Q]$, where $Q$ appears $n$ times. Since $G$ is an NS--group, if follows 
by Theorem 1.3 from \cite{AGG} that $Z(\Phi(G))$ is not contained in $Z_p(G)$. Therefore $ A(p) \neq 0$ and so 
$|[A,Q,Q]| \geq p^{p-1}$. On the other hand, the subgroup of $Q$--fixed points of $A$ is
not trivial, hence $|A^Q| \geq p$ and $|A^Q \otimes Q| \geq p^m$. We obtain $|Z(\Phi(G))| \geq p^{m+p}$
and by Proposition \ref{Frattini nonabelian}, we have $|\Phi(G)| \geq p^{m+p+2}$.
\end{proof}

In particular, we get a lower bound on the order of NS--groups depending on the
minimum number of generators.

\begin{corollary}
If $G$ is an NS--group with $d(G)=m$, then $|G| \geq p^{2m+p+2}$.
\end{corollary}

\begin{corollary}
\label{CorTamano}
If $G$ is a $p$-group with $|\Phi(G)| < p^{p+4}$ or $|G| < p^{p+6}$, then it is an S--group.
\end{corollary}

For $p=2$, this result was shown in \cite{Ab2} using GAP and later in \cite{Gh2} mathematically.
Note that ten $2$-generated NS--groups of order $2^8$ were found in \cite{Ab2} using GAP. The 
following proposition shows that, should it exist, an NS--group of order $p^{p+6}$ would behave similarly to those
groups.

\begin{proposition}
\label{PropiedadesNS}
Let $G$ be an NS--group of order $p^{p+6}$. Then we have:
\begin{itemize}
    \item[(a)] $d(G)=2$.
    \item[(b)] $|Z(\Phi(G))|=p^{p+2}$.
    \item[(c)] $Z(G)$ is cyclic of order $p$.
\end{itemize}
\end{proposition}

\begin{proof}
The proof of Theorem \ref{SmallFrattini} showed that $|Z(\Phi(G))| \geq p^{p+d(G)}$. 
If $d(G)\geq 3$, we would have $|Z(\Phi(G))| \geq p^{3+p}$ and then $[G:\Phi(G)] \leq p$. 
But then $\Phi(G)$ would be abelian and by Proposition \ref{Frattini nonabelian},
$G$ would be an S--group. Thus $d(G)=2$ and therefore $|\Phi(G)|=p^{p+4}$. The 
order of $Z(\Phi(G))$ must be exactly $p^{p+2}$, otherwise $\Phi(G)$ would be abelian. 

Let $A=Z(\Phi(G))$ and $Q=G/\Phi(G)$. Note that $A^Q=Z(G)$. If $|A^Q| > p$, then we 
would have $|Q\otimes A^Q|\geq p^2$ and so
\[ |A| = |A^Q| \cdot |A^Q \otimes Q| \cdot | [A,Q,Q] | \geq p^{p+3} \]
which contradicts part (b). Hence $Z(G)$ is cyclic of order $p$.
\end{proof}

We are particularly interested in the NS--groups of order $2^8$, and we see in this proposition that
for such groups, the order of the center of the Frattini subgroup is $2^4$. In the next section we 
will identify the isomorphism type of $Z(\Phi(G))$ for these groups.

\begin{remark}
NS--groups are natural choices when looking for counterexamples to Berkovich's conjecture. However,
one can find in GAP \cite{GAP4} that the ten NS--groups of order $2^8$ found in \cite{Ab2} do have noninner 
automorphisms of order $2$. We include below a table with the number of order-two noninner automorphisms
for each of these groups.
\begin{table}[h!]
	\centering 
	\renewcommand{\arraystretch}{1.2}
	  \begin{tabular}{ | c | c | c | c  | c | c | c  | c | c | c | c | }
	    \hline
	    IdSmallGroup$[256,i]$ & $298$ & $299$ & $300$ & $301$ & $302$ & $303$ & $304$ & $305$ & $306$ & $307$ 
	    \\ \hline
	    Order-$2$ noninner aut. & $576$ & $576$ & $512$ & $512$ & $576$ & $576$ & $576$ & $704$ & $704$ & $576$ 
	    \\ \hline
	  \end{tabular}
	\label{table : Comparison}
\end{table}
\end{remark}

\section{Cohomologically trivial faithful finite modules}
\label{sec : Faithful}

The results in the previous section led to lower bounds on the order of the center of
the Frattini subgroup of an NS--group. In this section we find further restrictions on
the isomorphism type of the center of the Frattini subgroup. This is achieved thanks to 
the following lemma, which places a constraint on the action of the Frattini quotient.

\begin{lemma}
Let $G$ be an NS--group. Then the action of $G/\Phi(G)$ on $Z(\Phi(G))$ is effective.
\end{lemma}

\begin{proof}
Note that $g\Phi(G)$ acts trivially on $Z(\Phi(G))$ if and only if $g \in C_G(Z(\Phi(G))$. 
But $C_G(Z(\Phi(G))=\Phi(G)$ by the corollary on page 2 of \cite{Sch}, hence $g\Phi(G)=\Phi(G)$
and so the action is effective.
\end{proof}

Now let $G$ be an NS--group. Since $G$ is not cyclic, we have $d(G)=r \geq 2$. Then $Z(\Phi(G))$ 
is an abelian $p$-group with an effective action of $(\Z/p)^r$, which is cohomologically trivial. 
We will exploit this fact to rule out some possibilities for $Z(\Phi(G))$. First we dismiss cyclic $p$-groups.

\begin{proposition}
\label{Cyclic}
Let $M$ be a cyclic $p$-group with an effective action by automorphisms of an elementary abelian $p$-group
of rank at least two. Then $M$ is not cohomologically trivial.
\end{proposition}

\begin{proof}
If $p$ is odd, the group $\Aut(\Z/p^k)$ is cyclic of order $p^{k-1}(p-1)$, so it does not have any elementary
abelian $p$-subgroup of rank two. Therefore $p=2$ and $M \cong \Z/2^k$ for some $k$. Since $\Aut(\Z/2)$ is trivial and $\Aut(\Z/4)$
is cyclic, we must have $k \geq 3$. In this case $\Aut(\Z/2^k) \cong \Z/2 \times \Z/2^{k-2}$ contains a unique
elementary abelian $2$-subgroup of rank two. Since this subgroup contains the automorphism that sends each element
to its inverse, the trace map is trivial and therefore $M$ is not cohomologically trivial.
\end{proof}

Next we exclude certain elementary abelian $p$-groups.

\begin{proposition}
\label{ElementaryAbelian}
Let $M$ be an elementary abelian $p$-group which is cohomologically trivial for the action by automorphisms of an elementary abelian 
$p$-group of rank $r \geq 2$. Then the rank of $M$ is a multiple of $p^r$.
\end{proposition}

\begin{proof}
If an elementary abelian $p$-group has an action by automorphism of $E=(\Z/p)^r$, then it can be regarded as an $\F_p E$--module. 
By Theorem VI.8.5 in \cite{Br}, it is cohomologically trivial if and only if it is free as a $\F_p E$--module. 
The result follows since $\F_p E \cong (\Z/p)^{p^r}$ as an abelian group.
\end{proof}

In order to eliminate groups of the form $\Z/p^k \times \Z/p$, we need the following auxiliary result.

\begin{lemma}
\label{Congruence}
Let $k\geq 2$. If there exists a nonnegative integer $r$ with $ 0 \leq r < 2^k$ and $r^2 - 1 \equiv 2^{k-1} \Mod 2^k$, then $k \geq 4$ and
\[ r \in \{ 2^{k-2} \pm 1, 3 \cdot 2^{k-2} \pm 1 \} \]
\end{lemma}

\begin{proof}
Note that $r$ must be odd, say $r=2n+1$. Then $r^2 = 4n^2+4n+1$, which is congruent to $1$ modulo $8$. Hence
there are no solutions to the desired congruence when $k$ equals two or three. Hence assume $k \geq 4$ from now on. 
It is easy to check that $2^{k-2} \pm 1$ and $3 \cdot 2^{k-2} \pm 1$ are solutions to the congruence. If $r$ is a
solution, we must have $(r-1)(r+1) =  2^{k-1} c$ for some $c$ odd. Then $r-1 = 2^am$ and $ r+1=2^bn$ with $mn=c$ and $a+b=k-1$. Then
\[ 2^a m+2 =2^b  n \]
Checking the parity we see that $a=0$ if and only if $b=0$, but this possibility is incompatible with $k \geq 2$. Therefore
$a$ and $b$ must be positive and so
\[ 2^{a-1}m+1 =  2^{b-1} n \]
Checking parity again, we find that either $a=1$ or $b=1$. If $a=1$, then $b=k-2$ and so $r= 2^{k-2}n-1$. If $b=1$, then $a=k-2$ 
and $ r = 2^{k-2}m+1$. Since $r<2^k$, the only possible values of $n$ and $m$ are one and three.
\end{proof}

For groups of the form $\Z/p^k \times \Z/p$, we will perform a delicate analysis of the elementary abelian $p$-subgroups of
their automorphism groups.

\begin{proposition}
\label{ProductNotCohomologicallyTrivial}
Let $k \geq 2$. If an elementary abelian $p$-group of rank at least two acts effectively on $M = \Z/p^k \times \Z/p$ by automorphisms,
then $M$ is not cohomologically trivial.
\end{proposition}

\begin{proof}
Let $q \colon \Z/p^k \to \Z/p$ be mod $p$ reduction and let $j \colon \Z/p \to \Z/p^k$ be the standard
inclusion. We have $qj=0$ and $jq = p^{k-1}$, where we are denoting by $n$ the automorphism $[x] \mapsto [nx]$.

We first treat the case $p=2$. Note that every automorphism of $\Z/2^k \times \Z/2$ has the form
\[ A = \left( \begin{array}{cc}
           \alpha & mj \\
           nq & 1 \end{array} \right) \]
where $\alpha \in \Aut(\Z/2^k)$ and $m, n \in \{0,1\}$. We refer to $\alpha$ as the corner of the automorphism. 
Since $\alpha j=j$ and $q\alpha = q$, we have
\[ A^2 = \left( \begin{array}{cc}
           \alpha^2+2^{k-1}mn & 0 \\
           0 & 1 \end{array} \right) \]
Hence $A^2=I$ if and only if $\alpha^2 + 2^{k-1}mn = 1$. If $\alpha([1]) = [r]$ with $ 0 \leq r < 2^k$, this holds
if and only if
\[ r^2 + 2^{k-1}mn \equiv 1 \Mod 2^k \]
If $mn=0$, then $r^2 \equiv 1 \Mod 2^k$ and it is well known that $ r \in \{ 1, -1, 2^{k-1}+1, 2^{k-1}-1 \}$. If $mn=1$, then $m=n=1$ and
\[ r^2 -1 \equiv 2^{k-1} \Mod 2^k \]
By Lemma \ref{Congruence}, we obtain $ k \geq 4$ and $ r \in \{ 2^{k-2} \pm 1, 3 \cdot 2^{k-2} \pm 1 \}$. Hence if $k \geq 4$, we have
automorphisms of order $2$ of the form
\[ \left( \begin{array}{cc}
           r & j \\
           q & 1 \end{array} \right) \]
We use the following homomorphism 
\begin{gather*}
t \colon \Aut(\Z/2^k \times \Z/2) \to \Z/2 \oplus \Z/2 \\
\left( \begin{array}{cc}
           \alpha & mj \\
           nq & 1 \end{array} \right) \mapsto ([m],[n]) 
\end{gather*}
to distinguish different types of automorphisms. More precisely, we say
that $t(A)$ is the type of $A$. We outline in the following table when 
these order-two automorphisms commute depending on their type
\begin{center}
\begin{tabular}{|c|c|c|c|c|}
\hline
& ([0],[0]) & ([1],[0]) & ([0],[1]) & ([1],[1])\\
\hline
 ([0],[0]) & Yes & Yes & Yes & Yes \\
\hline
 ([1],[0]) & Yes & Yes &  No & No \\
\hline
 ([0],[1]) & Yes & No & Yes & No \\
\hline 
 ([1],[1]) & Yes & No & No & Yes \\
\hline
\end{tabular}
\end{center}

Let us consider first the case of an efective action of $(\Z/2)^2$. The corresponding
subgroup of automorphisms is generated by two order-two commuting automorphisms, say
with corners $r$ and $r'$. From the table above and the fact that $t$ is a homomorphism,
we obtain that the trace map is represented by the matrix
\[ \left( \begin{array}{cc}
           1+r+r'+rr' & 0 \\
           0 & 0 \end{array} \right) = \left( \begin{array}{cc}
           (1+r)(1+r') & 0 \\
           0 & 0 \end{array} \right)\]
The element $([0],[1])$ is not in the image of the trace and it is fixed by order-two automorphisms 
of type $([0],[0])$ and $([0],[1])$, hence we only need to to check two cases.

Case 1: The subgroup is generated by an automorphism of type $([0],[0])$ and an automorphism of type $([1],[0])$.
If the automorphism of type $([0],[0])$ is $-1$, the trace map is trivial. If $r=2^{k-1}-1$, then
\[ (1+r)(1+r') = 2^{k-1} (1+r') \]
Since $r'$ is odd, the trace map is also zero. It is also zero when $r'=-1$ or $r'=2^{k-1}-1$.
When $r=2^{k-1}+1$ and $r' \in \{ 2^{k-1}+1, 1\}$, then
\[ (1+r)(1+r') = 2(2^{k-2}+1)(1+r') \]
which is a multiple of $4$. However, $([2],[0])$ is a fixed point.

Case 2: The subgroup is generated by an automorphism of type $([0],[0])$ and an automorphism of type $([1],[1])$.
This can only happen if $k \geq 4$ by Lemma \ref{Congruence}. Again we can assume that $r=2^{k-1}+1$, since the trace 
map vanishes otherwise. Then if $r' = 2^{k-2}n +1$, the element $([2],[1])$ is fixed and if $r' = 2^{k-2}n -1$, the 
element $([2^{k-2}],[1])$ is fixed. None of these elements belong to the image of the trace.

Now consider an effective action of an elementary abelian $2$-group of rank $r > 2$. If the corresponding subgroup
of automorphisms is generated by elements of type $([0],[0])$ and $([0],[1])$, the trace map is zero. Otherwise we need
to consider two cases again.

Case 1: The subgroup is generated by automorphisms of type $([0],[0])$ and type $([1],[0])$. If any of the generating 
automorphisms has corner $r=-1$ or $ r = 2^{k-1}-1$, the trace map is trivial. So we can assume that the group
is generated by the automorphism of type $([0],[0])$ with corner $r = 2^{k-1}+1$ and the automorphisms
of type $([1],[0])$ with corners $2^{k-1}+1$ and $1$. In this case we have
\[ (1+r)(1+r')(1+r'') = (2^{k-1}+2)(2^{k-1}+2) 2 \equiv 8 \Mod 2^k \]
Hence $([2],[0])$ is a fixed point which does not belong to the image of the trace.

Case 2: The subgroup is generated by automorphisms of type $([0],[0])$ and type $([1],[1])$. We can again assume that there is only one
automorphism of type $([0],[0])$ with corner $r = 2^{k-1}+1$. If there is at least one automorphism of type $([1],[1])$ with corner 
$ r' = 2^{k-2}n -1$, then
\[ (1+r)(1+r')(1+r'') = (2^{k-1}+2)  2^{k-2}n (1+r') \equiv 0 \Mod 2^k \]
and the trace map is zero. Hence we can assume that the group is generated by the automorphism of type $([0],[0])$ and
corner $r=2^{k-1}+1$ and the automorphisms of type $([1],[1])$ with corners $ r' = 2^{k-2}+1$ and $r'' = 3 \cdot 2^{k-2} + 1$. 
In this case the element $([2],[1])$ is fixed and not in the image of the trace map.

We now treat the case $p \neq 2$, where every automorphism of $\Z/p^k \times \Z/p$ has the form
\[ A = \left( \begin{array}{cc}
           \alpha & mj \\
           nq & \beta \end{array} \right) \]
with $\alpha \in \Aut(\Z/p^k)$, $\beta \in \Aut(\Z/p)$ and $m, n \in \{0,\ldots,p-1\}$. Recall that $\Aut(\Z/p^k) \cong \Z/p^{k-1}(p-1)$.
Since the homomorphism
\begin{gather*}
\Aut(\Z/p^k) \to (\Z/p)^{\times} \\
 \alpha \mapsto [\alpha(1)]
\end{gather*}
is surjective, a $p$-Sylow of $\Aut(\Z/p^k)$ is given by
\[ \{ \alpha \in \Aut(\Z/p^k) \mid \alpha(1) \equiv 1 \Mod p \} \]
Therefore a $p$-Sylow of $\Aut(\Z/p^k \times \Z/p)$ is given by
\[ S = \left\{ \left( \begin{array}{cc}
           \alpha & mj \\
           nq & 1 \end{array} \right) \in \Aut(\Z/p^k \times \Z/p) \: \Bigg{|} \: \alpha(1) \equiv 1 \Mod p \right\} \]
Note that these automorphisms $\alpha$ in the first entry of elements of $S$ fix $[p^{k-1}]$. Hence $\alpha j $ and $q\alpha =q$.
Given such an automorphism $A$ in $S$, by induction
\[ A^r = \left( \begin{array}{cc}
               \alpha^r + \frac{r(r-1)}{2}mnp^{k-1} & rmj \\
                   &   \\
               rnq & 1 \end{array} \right) \]
and in particular
\[ A^p = \left( \begin{array}{cc}
               \alpha^p & 0 \\
               0 & 1 \end{array} \right) \]
We see then that $A^p=1$ if and only if $\alpha^p=1$. Since $\Aut(\Z/p^k)$ has a unique subgroup
of order $p$, this happens if and only if $\alpha([1])=[ap^{k-1}+1]$ for some $ 0 \leq a \leq p-1$. 
For the corresponding $A$ we compute
\[ I + A + \ldots + A^{p-1} = \left\{ \begin{array}{ll}
                               \left( \begin{array}{cc}
                                      p & 0 \\ 
                                      0 & 0 \end{array} \right) & \text{ if } p \geq 5 \\
                                      \\
                               \left( \begin{array}{cc}
                                      3 + 3^{k-1} mn & 0 \\ 
                                      0 & 0 \end{array} \right) & \text{ if } p = 3 \end{array} \right. \]
Any elementary abelian $p$-subgroup of $\Aut(\Z/p^k \times \Z/p)$ is conjugate to a subgroup
of $S$. By the computation above, if the rank of this subgroup is at least two, then the image 
of the trace map is contained in $p^2\Z/p^k \times \{ 0 \}$. But $([p],[0])$ is always a fixed point, 
hence the module is not cohomologically trivial.
\end{proof}

We arrive at the main result of this section by analyzing $\Z/4 \times \Z/4$.

\begin{theorem}
Let $M$ be a finite abelian $2$-group of order at most $2^4$ which is cohomologically trivial for the 
effective action of an elementary abelian $2$-group of rank at least two. Then $ M $ is isomorphic to $(\Z/2)^4$ or $\Z/4 \times (\Z/2)^2$.
\end{theorem}

\begin{proof}
By Propositions \ref{Cyclic}, \ref{ElementaryAbelian} and \ref{ProductNotCohomologicallyTrivial}, it suffices to show that $\Z/4 \times \Z/4$
can not be a cohomologically trivial effective $E$--module for any elementary abelian $2$-group $E$ of rank at least two.

Let us regard the elements of $\Aut(\Z/4 \times \Z/4)$ as elements of $M_{2\times 2}(\Z/4)$. Consider the subgroup 
$2(\Z/4 \times \Z/4) \cong \Z/2 \times \Z/2$. The restriction
\[ \res \colon \Aut(\Z/4 \times \Z/4) \to \Aut(\Z/2 \times \Z/2) \]
has kernel
\[ K = \{ I + 2B \mid B \in M_{2\times 2}(\Z/4) \} \cong (\Z/2)^4 \]
It fits into a short exact sequence
\[ 1 \to (\Z/2)^4 \to \Aut(\Z/4 \times \Z/4) \to \Sigma_3 \to 1 \]
Let $A$ be the subgroup of $\Aut(\Z/2 \times \Z/2)$ generated by the automorphism $\sigma$
that permutes the coordinates, and let $S = \res^{-1}(A)$. The subgroup $S$ is a
$2$-Sylow of $\Aut(\Z/4 \times \Z/4)$ and it is a semidirect product $(\Z/2)^4 \rtimes \Z/2$.
Hence we may assume that $E$ is a subgroup of $S$. An element $I+2B$ commutes with $(I+2B')\sigma$ 
if and only if $2B$ has the form
\[ \left( \begin{array}{cc}
           a & b \\
           b & a \end{array} \right) \]
Similarly, the element $(I+2B')\sigma$ has order two if and only if $2B'$ has the same form.  

Assume first $E$ is generated by elements $I+2B$ and $I+2B'$. Then the trace map is given by
\[ I + (I+2B) + (I+2B') + (I + 2B + 2B') = 0 \]
and therefore $M$ is not cohomologically trivial. If $E$ is generated by elements $I+2B$ 
and $(I+2B')\sigma$, the trace map has the form
\begin{align*}
\tau & = I + (I+2B) + (I+2B')\sigma + (I+2B+2B')\sigma \\
     & = 2I + 2B + 2\sigma + 2B\sigma \\
     & = 2(I+B)(1+\sigma) \\
     & = \left( \begin{array}{cc} 
                2+a+b & 2+a+b \\
                2+a+b & 2+a+b \end{array} \right)
\end{align*}
If $a+b=2$, the trace map is trivial and $M$ is not cohomologically trivial. If $a+b=0$, the
image of $\tau$ is the subgroup generated by $(2,2)$. Note that in this case $a=b=2$ and so
\[ I + 2B = \left( \begin{array}{cc}
                    3 & 2 \\
                    2 & 3 \end{array} \right) \]
On the other hand,
\[ (I+2B')\sigma = \left( \begin{array}{cc}
                           d-1 & 1+c \\
                           1+c & d-1 \end{array} \right) \]
for certain elements $c$, $d \in \{ 0 , 2 \}$. If $d-1=1+c$, then $(1,3)$ is fixed by both matrices.
If $d-1 \neq 1+c$, then $(1,1)$ is fixed by both matrices. In any case, $M$ is not cohomologically trivial.

Finally, if the rank of $E$ is greater than two, it must have an elementary abelian $2$-subgroup
that is contained in $K$ and therefore the trace map is trivial.
\end{proof}

Recall from the previous section that for an NS--group $G$ of order $2^8$, the center of
its Frattini subgroup must have order $2^4$. The following corollary is immediate
from the previous theorem.

\begin{corollary}
If $G$ is an NS--group of order $2^8$, then $Z(\Phi(G))$ is isomorphic to 
$\Z/4 \times (\Z/2)^2$ or $(\Z/2)^4$.
\end{corollary}

When we check the list of ten NS--groups of order $2^8$ using GAP, we see that the center
of the Frattini subgroup is isomorphic to $(\Z/2)^4$ for six of them and $\Z/4 \times (\Z/2)^2$
for the rest, so this result is the best possible.

\begin{corollary}
If $G$ is an NS--group of order $2^8$, then $Z(\Phi(G))/Z(G) \cong (\Z/2)^3$.
\end{corollary}

\begin{proof}
The center of an NS--group $G$ of order $2^8$ is isomorphic to $\Z/2$ by Proposition \ref{PropiedadesNS}.
Since $\Z/2 \times \{ 0 \}$ is a characteristic subgroup of $\Z/4 \times \Z/2$, we see that if $Z(\Phi(G)) \cong \Z/4 \times \Z/2$, 
then the inclusion of $Z(G)$ in $Z(\Phi(G))$ corresponds to the inclusion of $\Z/2 \times \{0\}$ in $\Z/4 \times \Z/2$. 
\end{proof}

\section{Metacyclic groups and groups of small coclass}
\label{sec : SmallCoclass}

Motivated by previous works on Berkovich's conjecture for groups of small coclass, in this
section we study Schmid's conjecture for $2$-groups of coclass at most two and nonabelian
metacyclic $p$-groups. The case of groups of coclass one follows easily from a result in \cite{Ab2}.
	
\begin{proposition}
\label{CoclassOne}
Let $G$ be a nonabelian $p$-group of maximal nilpotency class. Then $G$ is an S--group.
\end{proposition}

\begin{proof}
If the order of $G$ is $p^3$, then the nilpotency class of $G$ is two, and it follows
by Theorem 3.6 in \cite{Ab1}. Assume now that $|G| > p^3$. Since $G$ has maximal nilpotency 
class, we have that $Z_2(G)/Z(G)$ is cyclic and then
\[ d(G) d(Z(G)) \neq d \left( \frac{Z_2(G)}{Z(G)} \right) \]
since $d(G) \geq 2$. We conclude that $G$ is an S--group by Lemma 3.8 in \cite{Ab2}.
\end{proof}
 
For the next result, we were inspired by the article \cite{AGGRW}, where Berkovich's
conjecture was established for $p$-groups of coclass two.
 
\begin{proposition}
\label{CoclassTwo}
Let $G$ be a finite $2$-group of coclass two. Then $G$ is an S--group. 
\end{proposition}

\begin{proof}
Suppose that $G$ is an NS--group of order $2^n$. We may assume $n \geq 4$, otherwise the result holds by
Theorem 1.2 in \cite{AGG}. By Lemma 3.8 in \cite{Ab2}, we have that $d(Z_2(G)/Z(G))=d(G)d(Z(G))$. 
Since $G$ is of coclass two, the order of $Z_2(G)/Z(G)$ is at most four, hence $d(Z_2(G)/Z(G)) \leq 2$.
The group $G$ is not cyclic, thus $d(G)=2 = d(Z_2(G)/Z(G))$ and therefore $Z(G)$ is cyclic of order
$2$, the group $Z_2(G)/Z(G)$ is elementary abelian of rank two and $Z_2(G)$ is noncyclic of order $2^3$. 
Therefore $Z_{n-2}(G) = \Phi(G)$.

The elements of $Z_2(G)$ commute with commutators. Since $Z_2(G)/Z(G)$ is elementary abelian, we have
$x^2 \in Z(G)$ whenever $x\in Z_2(G)$. Then for all $g \in G$, we have 
\[ 1 = [x^2,g] = [x,g]^2 = [x,g^2] \]
and therefore the elements of $Z_2(G)$ commute with the elements of $\Phi(G)$. Thus $Z_2(G) \leq Z(\Phi(G))$.
Under these circumstances, the case $p=2$ of the proof of Theorem 3.1 in \cite{AGGRW} shows that $[G,G]$ is cyclic. 
Note that $|\Phi(G)/[G,G]| = |Z_{n-2}(G)/[G,G]| \leq 2$. By Proposition \ref{Frattini nonabelian} we deduce that 
this order must be $2$, hence $[G,G]$ is a maximal subgroup of $\Phi(G)$. By the classification of $2$-groups with 
a cyclic maximal subgroup (see Result 5.3.4 in \cite{Ro}, for instance), the Frattini subgroup $\Phi(G)$ must be abelian or have cyclic 
center. Hence the result follows by Propositions \ref{Frattini nonabelian} and \ref{Cyclic}. 
\end{proof}

\begin{remark}
By Theorem 1.1 in \cite{AGGRW}, any finite $p$-group $G$ of coclass $2$ has a noninner automorphism of order $2$ that fixes every element
of the center. However, this is not enough to guarantee that $G$ is an S--group. 
\end{remark}

\begin{remark}
Using GAP, one can see that a group $G$ of order $2^8$ is an NS--group if and only if its nilpotency class is four and 
$Z(\Phi(G))/Z(G) \cong (\Z/2)^3$. By Theorem 1.2 in \cite{AGG} and the results in this article,
we can say that an NS--group $G$ of order $2^8$ has nilpotency class four or five and must
satisfy $Z(\Phi(G))/Z(G) \cong (\Z/2)^3$. Only the case of groups of order $2^8$ with coclass
three needs to be dismissed to achieve one of the implications in this equivalence. 
\end{remark}

We end this section with an analysis of metacyclic $2$-groups.

\begin{proposition}
\label{theorem 11.3}
Let $G$ be a nonabelian metacyclic $2$-group. Then $G$ is an S--group. 
\end{proposition}

\begin{proof}
Note that if $G$ has a cyclic maximal subgroup, then $\Phi(G)$ is abelian, hence $G$ is an S--group by Proposition \ref{Frattini nonabelian}. 
Now assume that $G$ is an NS--group, in particular it does not have a maximal cyclic subgroup. By Theorem 2.2 in \cite{XZ}, the group 
$G$ is generated by two elements $a$, $b$ such that $[G,G]$ is generated by a power of $a$. 

Let us denote $A=Z(\Phi(G))$ and $Q=G/\Phi(G)$. Consider the cyclic subgroup $H_g$ of $Q$ generated by a nontrivial element $g\Phi(G)$. The trace map 
for the action of $H_g$ on $A$ is given by
\[\tau_g(x)=xx^g=xg^{-1}xg=x^2[x,g] = x^2 a^m\] 
for some $m$, where the last equality holds because $[G,G]$ is generated by a power of $a$. Since $G$ is an NS--group, the image of $\tau_g$ equals
$A^{H_g}$ and by Proposition 1 in \cite{Sch}, we have $C_Q(A^{H_g})=H_g$. In particular $a\Phi(G) \in C_Q(A^{H_a})$ and therefore $a$ commutes with
$x^2$ for all $x \in A$. Then $a$ commutes with the image of $\tau_b$, hence $a\Phi(G)$ belongs to $C_Q(A^{H_b})=H_b$, which is a contradiction. Thus
$G$ is an S--group.
\end{proof}

Note that for $p>2$, nonabelian metacyclic $p$-groups are regular groups (by Theorem 9.8(a) and Theorem 9.11 from \cite{B1}), hence S--groups 
by the main theorem in \cite{Sch}. Therefore this proposition settles the conjecture for all nonabelian metacyclic $p$-groups.

\begin{theorem}
Nonabelian metacyclic $p$-groups are S--groups. 
\end{theorem}

We now give an example of a nonabelian metacyclic $2$-group for which the criteria to be an S--group developed in this paper and other articles do not apply, to
emphasize that Proposition \ref{theorem 11.3} is finding new S--groups. 

The automorphism $\varphi$ of $\Z/32$ which sends $[x]$ to $[11x]$ has order eight. Consider 
the semidirect product $ G = \Z/32 \rtimes \Z/16$ where the standard generator of $\Z/16$ acts via $\varphi$. This is certainly a nonabelian metacyclic $2$-group. 
Let $a=([1],[0])$ and $b=([0],[1])$. Then
\[ (ab)^{16} = ([1+11+\ldots+11^{15}],[0]) = ([0],[0]) \]
where the last equality holds because
\[ 1+11+\ldots+11^{15} = 2(1+11+\ldots+11^7) = \frac{11^8-1}{5} = 0 \Mod 32 \]
since $\varphi$ has order eight. However, $a^{16} b^{16} = a^{16} = ([16],[0]) \neq (ab)^{16}$, hence $G$ is not semi-abelian. Therefore Theorem 1.1 from \cite{BG}
does not apply to $G$. The action of $2\Z/32$ preserves the subgroup $2\Z/16$ and the subgroup generated by both is a normal subgroup $N$
of $G$. It is easy to check that $G/N$ is elementary abelian of rank two, hence $N = \Phi(G)$. But $\Phi(G)$ is not abelian since the 
action of $2\Z/16$ only fixes $4\Z/32$. Therefore Proposition \ref{Frattini nonabelian} does not apply. Note also that $|G|=2^9$ and
$|\Phi(G)| = 2^7$, so Corollary \ref{CorTamano} does not apply either. Finally, since $G$ is a semidirect product, $[G,G]$ is generated
by elements of the form $ a \varphi(a)^{-1} $ with $a \in \Z/32$. Then $[G,G]=2\Z/32$. We find the rest of terms in the lower central series
in the same way:
\[ \{ 1 \} < 16\Z/32 < 8\Z/32 < 4\Z/32 < 2\Z/32 < G \]
hence the nilpotency class of $G$ equals five and its coclass equals four. Therefore Theorem 1.3 in \cite{AGG} and Propositions \ref{CoclassOne}
and \ref{CoclassTwo} do not apply to $G$. On the other hand, groups of order $2^9$ can be shown to be S--groups using GAP (see \cite{AGG}), but
this has not been proved mathematically yet.

\bibliographystyle{amsplain}
\bibliography{mybibfile}

\end{document}